\documentclass[12pt]{cpc}
\newcommand{\classoption}{cpc}

\newcommand{\coloroption}{color}

\usepackage{amssymb}
\usepackage{amsmath}
\usepackage{graphicx}
\usepackage{ifthen}
\usepackage{hyperref}
\usepackage{cite}
\def\ev#1{\langle#1\rangle}

\ifthenelse{\equal{\classoption}{article}}
{
  \usepackage{sectsty}
  \sectionfont{\centering\normalsize\mdseries\textbf}
  \numberwithin{equation}{section}

  \usepackage{amsthm}
  \theoremstyle{definition}
  \newtheorem{theorem}{Theorem}
  \theoremstyle{remark}
  \newtheorem{lemma}{Lemma}

  \textwidth = 6.5 in \textheight = 9 in
  \oddsidemargin = 0.0 in \evensidemargin = 0.0 in
  \topmargin = -0.3 in \headheight = 0.0 in \headsep = 0.0 in
  \parskip = 0.15 in \parindent = 0.0in
}
{
  \usepackage{times}
}

\begin{document}

\ifthenelse{\equal{\classoption}{article}}
{
  \title{A computation of the expected number of posts \\ in a finite random
    graph order}

  \author{Luca Bombelli$^{1,}$\footnote{bombelli@olemiss.edu}, Itai
    Seggev$^{2,}$\footnote{iseggev@knox.edu},      and 
    Sam Watson$^{1,}$\footnote{sswatson@olemiss.edu}\\  \\
    $^1$Department of Physics and Astronomy, University of Mississippi,\\
    108 Lewis Hall, University, MS 38677, U.S.A.\\ \\
    $^2$Department of Mathematics, Knox College\\
    Box 80, 2 East South Street, Galesburg, IL 61401, U.S.A.}
  \date{Draft, August 7, 2008}
}
{
  \title[Posts in a Finite Random Graph Order]{A Computation of the Expected
    Number of Posts in a Finite Random Graph Order}
  \author[L. Bombelli, I. Seggev, and S. S. Watson]{
    \spreadout{LUCA BOMBELLI}$^1$, \spreadout{ITAI SEGGEV}$^2$, and
        \spreadout{SAM WATSON}$^1$\\
    \affilskip{$^1$}Department of Physics and Astronomy, University of
          Mississippi,\\ 
    \affilskip 108 Lewis Hall, University, MS 38677, U.S.A.\\
    (email: \texttt{bombelli@olemiss.edu}, \texttt{sswatson@olemiss.edu})\\
    \affilskip{$^2$}Department of Mathematics, Knox College\\
    \affilskip Box 80, 2 East South Street, Galesburg, IL 61401, U.S.A.\\
    (email: \texttt{iseggev@knox.edu}) 
  }
  \date{August 7, 2008}
}

\maketitle

\begin{abstract}
\noindent 
A random graph order is a partial order achieved by independently
sprinkling relations on a vertex set (each with probability $p$) and adding
relations to satisfy the requirement of transitivity.  A \textit{post} is
an element in a partially ordered set which is related to every other
element. Alon et al.\ \cite{Alon} proved a result for the average number of
posts among the elements $\{1,2,\ldots,n\}$ in a random graph order on
$\mathbb{Z}$.  We refine this result by providing an expression for the
average number of posts in a random graph order on $\{1,2,\ldots,n\}$,
thereby quantifying the edge effects associated with the elements
$\mathbb{Z}\backslash\{1,2,\ldots,n\}$.  Specifically, we prove that the
expected number of posts in a random graph order of size $n$ is
asymptotically linear in $n$ with a positive $y$-intercept.  The error
associated with this approximation decreases monotonically and rapidly in
$n$, permitting accurate computation of the expected number of posts for
any $n$ and $p$. We also prove, as a lemma, a bound on the difference
between the Euler function and its partial products that may be of interest
in its own right.
\end{abstract}

\section{Introduction}
Several definitions of random partial orders can be found in the
combinatorics literature. If the number $n$ of elements of the underlying
set is fixed, perhaps the most natural definition is that of ``uniform
random order," in which we pick a member of the set of $n$-element posets
uniformly at random.  Although no practical way is known of generating
posets according to this definition for large $n$, it is known that as $n
\to \infty$ most of them are ``3-level" posets \cite{KR}. In this case,
increasing $n$ leads to posets with a greater width but not a greater
height, on average, because of the growth in the number of relations per
element. A second definition is that of ``random $k$-dimensional order",
for some integer $k$, in which one picks $k$ linear orders on $n$ elements
uniformly at random (in other words, $k$ randomly chosen permutations of
the set $\{1, 2, ..., n\}$), and then takes their intersection. The
resulting posets are of dimension $k$ by construction, and some of their
properties are known \cite{Wink}.

A third definition, and the one we will mainly be interested in here, is
that of ``random graph order'' which depends on a parameter $p \in (0,1)$.
To obtain a partial order of this type, one first generates a random graph
on the vertex set $\{1, 2, ..., n\}$ by including an edge $(i,j)$ with
probability $p$ for each pair of vertices $i$ and $j$; one then turns the
graph into a directed one by converting each edge $(i,j)$ into a relation
$i \prec j$ in the partial order if $i < j$ (in the usual order on the
integers); and, finally, one imposes transitive closure by adding relations
so that $i \prec k$ whenever there exists a $j$ such that $i \prec j$ and
$j\prec k$. Several properties of random graph orders, such as width,
height, and dimension have been studied \cite{AF}. In particular, it is
known \cite{AF} that the expected height of a random graph order on $n$ elements
grows linearly with $n$.

In the physics literature, random graph orders are also known as
``transitive percolation" because they arise in a special case of the
theory of directed percolation \cite{NS}, where non-local bonds in a
1-dimensional lattice are turned on with probability $p$. They also play a
prominent role among the stochastic sequential growth models that have been
proposed for the classical version of the dynamics of causal sets
\cite{RiSo}, and this is the context that most directly motivates our work.
A causal set \cite{BLMS} is a partially ordered set that is locally finite,
meaning that the interval or Alexandrov set $I(i,j):= \{l\mid i\prec l\prec
j\}$ is finite for every pair with $i\prec j$. In the
causal set approach to quantum gravity (for a recent review, see Ref.\ 
\cite{Joe}), the poset is seen as a discrete spacetime.  The partial 
order corresponds to the causal relations among its elements, and ``$i \prec
j$" can be read as ``$i$ causally precedes $j$", while volumes of spacetime
regions correspond to the cardinality of the appropriate subsets of the
poset.  The final theory is expected to assign a spacetime volume of the
order of $\ell_{\rm P}^4$ to each element, where $\ell_{\rm P}:=
\sqrt{G\hbar/c^3} = 1.6\times10^{-33}$ cm is the Planck length.

A {\em post\/} in a poset is an element that is related to every other
element in the poset. In other words, each post $n$ divides the ordered set
into the subset of elements that precede it, its ``past", and the subset of
elements that follow it, its ``future".  In the causal set interpretation
the spacetime ``pinches off" at $n$; this can be seen as the
zero-spatial-volume singularity at the end of a collapsing phase for the
universe and the beginning of a new expanding phase. Thus, a first set of
desirable conditions for transitive percolation to be considered as a
physically reasonable way of generating discrete spacetimes is that if a
random graph order develops multiple posts, the number of elements between
two posts be allowed to grow sufficiently large for that region to be able
to model our observed universe.

It has been known for some time \cite{Alon} that infinite random graph
orders have an infinite number of posts.  However, the occurrence of posts
in finite random graph orders has not been studied as extensively. We begin
by revisiting random graph orders on $\mathbb{N}$ and then analyzing finite
posets.  Roughly speaking, we find in our analysis that ``edge effects''
are small but non-negligible.  In the infinite case, the probability that
element $n$ is a post approaches a constant value fairly rapidly.  In a
finite case, the expected number posts is well approximated by this
limiting probability times the size of the set plus a small positive
offset.  We illustrate our conclusions with numerical simulations.

\section{Infinite Random Graph Orders}

We begin by calculating the probability that any given element in an
infinite random graph order is a post.  In order to express this
probability succinctly we define $q=1-p$,
\[\lambda_k(q)=(1-q)(1-q^2)\cdots(1-q^{k}), \quad {\rm and  } \quad
\kappa(q)=\lambda_\infty(q)=\lim_{k\rightarrow\infty}\lambda_k(q).\] The
function $\kappa(q)$ is known as the Euler function, and it has been
studied in considerable detail \cite{sokal}. In particular, $\kappa(q)>0$
for all $0<q<1$.  We now prove a theorem for one-way infinite random graph
orders similar to the result of Alon, et al.\ \cite{Alon} that the
probability that an element in a random graph order on $\mathbb{Z}$ is a
post is $\kappa^2(q)$.

\begin{theorem}\label{infty_thm} 
  In a random graph order on $\mathbb{N}$ with probability $0<p<1$,
  the probability that any given element $k$ is a post is given by
\begin{equation}
\begin{array}{rcl}
\Pr(k\mbox{ is a post})
&=& \lambda_{k-1}(q)\,\lambda_\infty(q)\\
&\gtrsim& \kappa^2(q) \label{postprob},
\end{array}
\end{equation}
where by $\gtrsim$, we mean ``greater than and, for large $k$, approximately
equal to.''
\end{theorem}

\begin{figure}
\begin{center}
\ifthenelse{\equal{\coloroption}{color}}
{\includegraphics[angle=0,width=\textwidth]{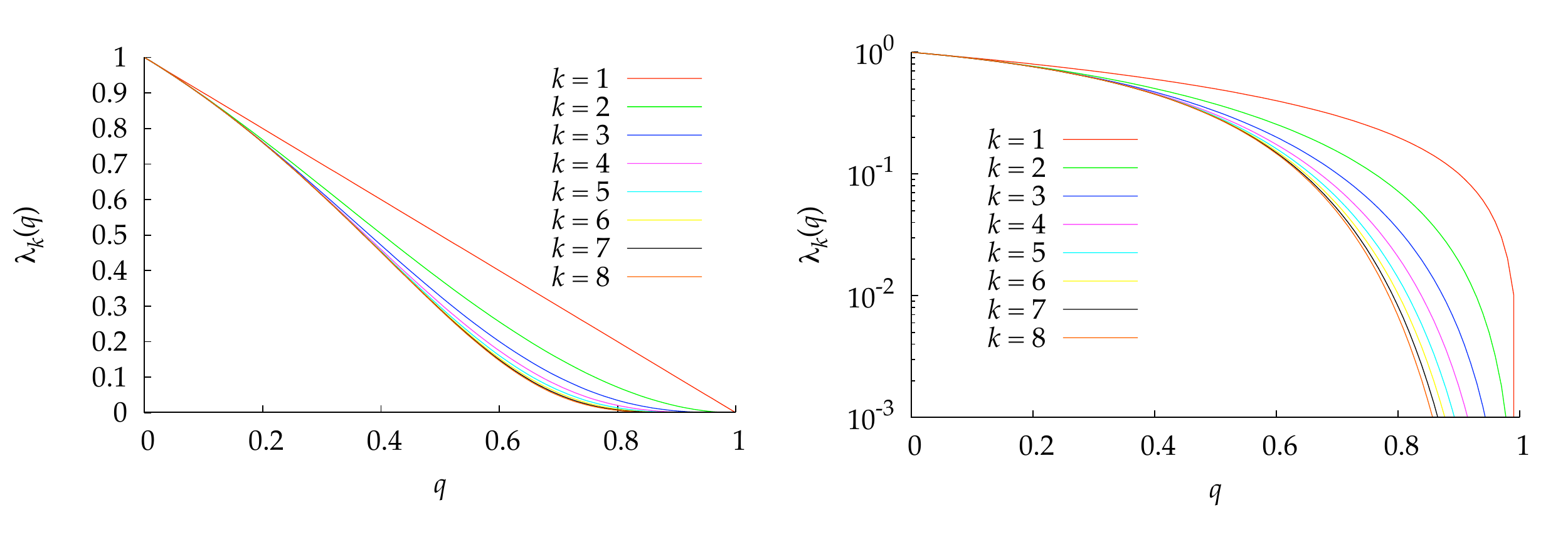}}
{\includegraphics[angle=0,width=\textwidth]{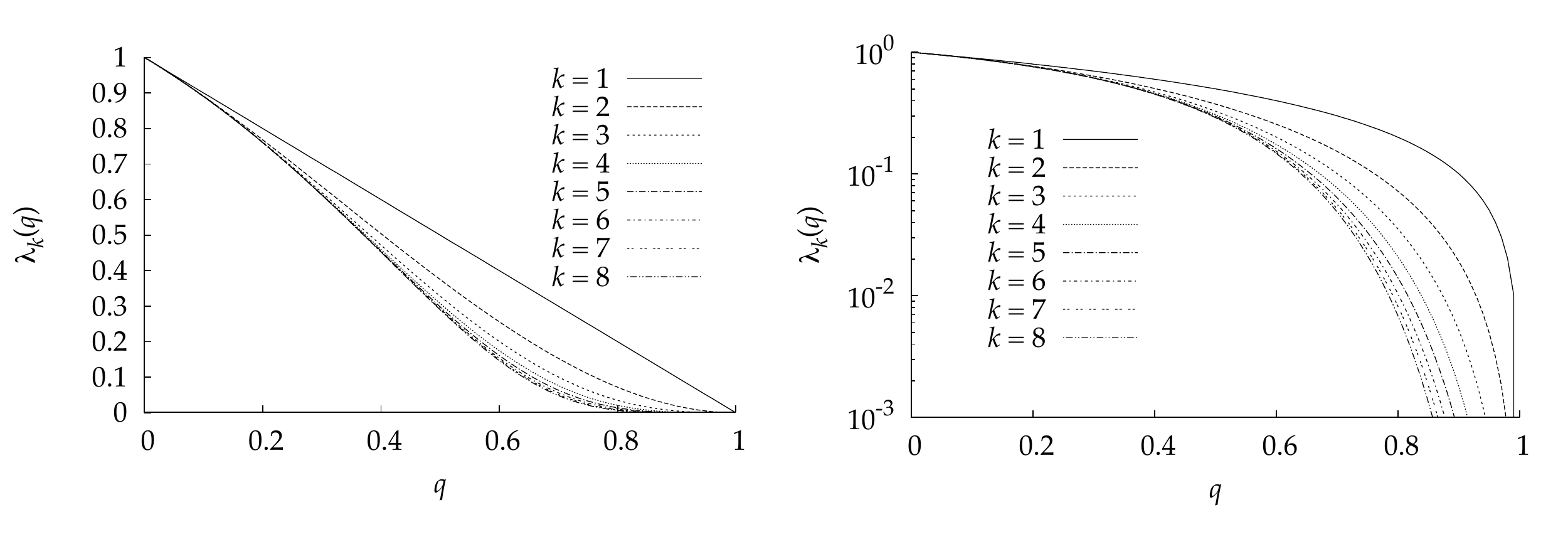}}
\caption{\label{f1} 
  Plot of $\lambda_k(q)$ as a function of $q$ for $k = 1,...,8$ 
  (from top to bottom).
}
\end{center}
\end{figure}

Plots of $\lambda_k(q)$ as functions of $q$ for various values of $k$ (see
Fig.\ \ref{f1}) illustrate the rate of the convergence of $\lambda_k(q)$ to
$\kappa(q)$.  Because of this convergence, the similarity in
(\ref{postprob}) holds for ``most'' $k\in\mathbb{N}$.  This observation
suggests two intuitively plausible results.  The first is that with unit
probability, there are infinitely many posts in any random graph order.
This result was first proved by Alon et al.\ \cite{Alon}.  Although the
original result was for random graph orders on $\mathbb{Z}$, the proof is
easily adapted to partial orders on $\mathbb{N}$.  The second result is
that the mean spacing between posts is $\kappa^{-2}(q)$.  Equivalently, the
expected number of posts after $N$ stages is well-approximated by $\kappa^2 N$.
This information on the structure of random 
graph orders is of the type we are interested in from the point of view of
their possible applications as models of discrete spacetimes, and in the
next section we will consider it in more detail.

\begin{proof}[Proof of theorem.]
  If $k$ is related to $k-1$, $k-2$, $\ldots$, $k-i+1$ (for $1<i<k$),
  then the probability that $k-i \nprec k$ is $q^i$, since by transitivity
  the only way for this to happen is for $k-i$ to be unrelated to each of
  $k-i+1$, $k-i+2$, ..., $k$.  Hence we find that
\[
\Pr(k-i \prec k \mid k-1 \prec k \;\wedge\; k-2 \prec k \;\wedge \ldots
\wedge\; k-i+1 \prec k) = 1-q^i,
\] 
where we have used the notation $\Pr(A|B)$ for the conditional probability
of $A$ given $B$ and $\wedge$ for logical and.  Repeatedly decomposing the
probability that $k$ is related to each element before it yields the
expression
\begin{align}
\nonumber 
   \Pr(k\text{ related to all previous }m )&=\Pr(k-1\prec k)\cdot\Pr(k-2\prec k
   \mid k-1 \prec k) \cdot \ldots \\ 
\nonumber &\hspace{1.6cm}  \cdot \Pr(1\prec k \mid 2 \prec k \; \wedge \ldots
   \wedge\; k-1\prec k) \\ 
\nonumber &= (1-q)(1-q^2)\cdots(1-q^{k-1}) \\
\nonumber &= \lambda_{k-1}(q) \\
&> \kappa(q). \label{lambda} 
\end{align}
The inequality follows because the partial products of $\lambda_\infty(q)$
are strictly decreasing.

On the other hand, the same logic that led to (\ref{lambda}) shows the
probability that $n$ is related to every later element is given by
\[ 
\begin{array}{rcl}
\Pr(n \text{ related to every later } m) &=& \Pr(n \prec n+1) \cdot 
\Pr (n \prec n +2 \mid n \prec n+1) \cdot \ldots\\
&=&  \prod_{j=1}^{\infty} (1-q^j)\\
&=& \kappa(q). 
\end{array}
\]
Moreover, the events $m \prec n$ for $m < n$ and $n \prec m$ for $m > n$
are independent.  If transitive closure were to relate $m < n$ and $k > n$
in a manner involving $n$, then $n$ would be the middle element and both
relations $m \prec n$ and $n \prec k$ would exist \textit{a priori}.  Hence
the event ``$n$ is a post'' will occur if and only if the two preceding,
independent events occur, which has probability
\begin{align*}
\Pr(n\mbox{ is a post})
&= \Pr(n\mbox{ related to all }m<n)\cdot \Pr(n\mbox{ related to all }m>n) \\
&= \lambda_{n-1}(q)\,\lambda_\infty(q) \\
&\gtrsim \kappa^2(q),
\end{align*}
as desired.
\end{proof}

\section{Posts in Finite Random Graph Orders}
From the results for infinite graph orders we expect that, to a good
approximation, the expected number of posts in an $n$-element poset is
$N_{\rm posts}(n) = \kappa^{2}(q)\,n$. In fact, the mean number of posts
among the elements $\{1,2,\ldots,n\}$ in a random graph order on
$\mathbb{Z}$ is equal to $\kappa^{2}(q)\,n$ \cite{BB}.  However, this number
should be smaller than the expected number of posts in a random graph order
on $\{1,2,\ldots,n\}$---and appreciably smaller for small $n$---because elements
near the edge are significantly likely to be related to all the elements in
$\{1,2,\ldots,n\}$ but not all the elements in $\mathbb{Z}$.  We have carried
out numerical simulations of transitive percolation with different values
of $p$ and $n$; Fig.\ \ref{f2} shows the resulting values of $N_{\rm
  posts}$ plotted versus $n$ for a fixed $p$. From this plot, one may see
that for large $n$ the number of posts is well approximated by a line with
a small offset. We will now prove the following theorem.

\begin{figure}
\begin{center}
\ifthenelse{\equal{\coloroption}{color}}
{\includegraphics[angle=0,width=\textwidth]{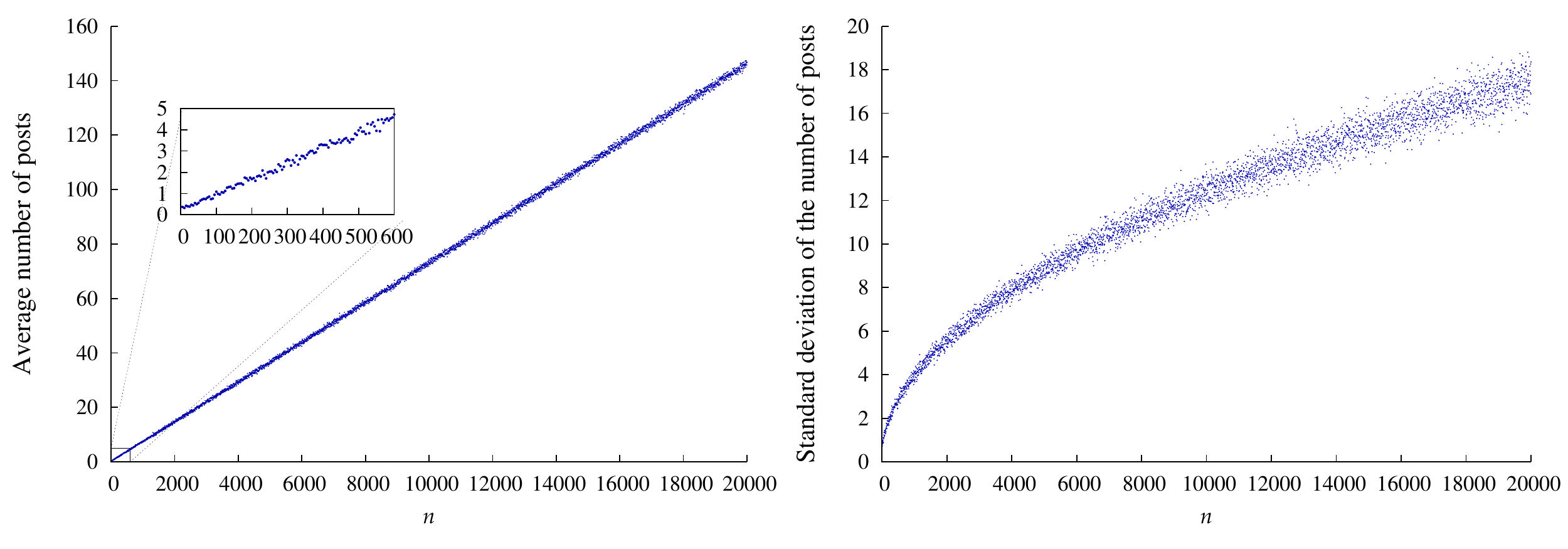}}
{\includegraphics[angle=0,width=\textwidth]{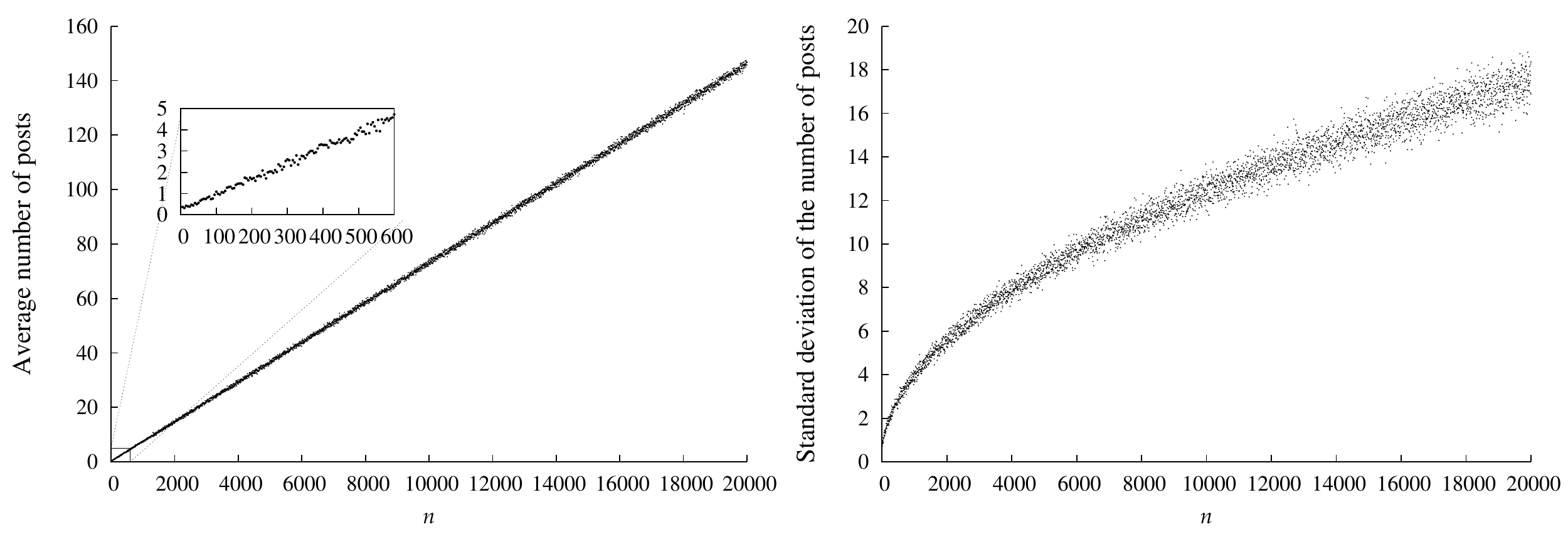}}
\caption{\label{f2} Average and standard deviation of the number of posts
  in a random graph order versus number of elements, for $p = 0.35$. For
  each value of $n$, four hundred $n$-element causal sets were generated, 
  and the 
  average and standard deviation of the numbers of posts are shown. The
  sampled values of $n$ are the multiples of 5 up to 20,000.  Notice in the
  inset graph that the average number of posts seems to be well fit by a
  line with positive $y$-intercept. }
\end{center}
\end{figure}

\begin{theorem}\label{finite_thm} 
  For all $0<q<1$, there exists a sequence of real numbers
  $\{b_n(q)\}_{n=1}^{\infty}$ so that for all $n\geq1$, $b_n(q)$ is
  strictly between 0 and 1 and the expectation value $\langle
  N_{n,q}\rangle$ of the number of posts in a transitively percolated
  causal set on $\{1,\, 2,\, 3,\, \ldots,\, n\}$ with probability $p=1-q$
  satisfies 
  \begin{equation} \label{expect}
    \langle N_{n,q}\rangle = \kappa^2(q)\cdot n+b_n(q)\,.
  \end{equation}
  Moreover, $\{b_n(q)\}_{n=2}^{\infty}$ is strictly monotonically
  decreasing to a positive limit $b(q)$ given by the expression
  \begin{equation} \label{b_lim}
    b(q) = 2\kappa(q)\sum_{k=0}^{\infty}(\lambda_k(q)-\kappa(q)).
  \end{equation}
\end{theorem}

For notational convenience, we will drop the explicit dependence of
$\kappa^2$ and $b_n$ on $q$.  We also introduce the abbreviations
\[\mu_{n}:= \prod_{i=n}^{\infty}(1-q^i),\qquad
S_n:= \sum_{k=1}^{\infty}\frac{q^{nk}}{\lambda_k}.\]
Notice that $\mu_1=\kappa$.
In order to prove the theorem, we first establish the following estimates.

\begin{lemma} \label{finite_lemma} For all $n\geq2$, we have 
\begin{equation}\label{lemmaclaim}
(\lambda_{n-1}-q^n) S_n < q^n, 
\end{equation}
and
\begin{equation}\label{mainlemmaclaim}
\lambda_{n-1}-\kappa < q^n.
\end{equation}
\end{lemma}

\begin{proof} If $\lambda_{n-1}-q^n$ is not positive, then
  (\ref{lemmaclaim}) clearly holds since $S_n$ and the right-hand side are
  positive.  So suppose that $\lambda_{n-1}-q^n$ is positive.  Because the
  $\lambda_k$ are monotonically decreasing in $k$ and $n-1\geq1$, we may
  replace $\lambda_{n-1}$ with $\lambda_1$ to find
\begin{align*}
(\lambda_{n-1}-q^n)S_n&
  ~\leq~(\lambda_{1}-q^n)\left(\sum_{k=1}^{\infty}\frac{q^{nk}}{\lambda_k}\right).
\end{align*}
Distributing and using the definitions of $\lambda_k$ gives
\begin{align*}
(\lambda_{1}-q^n)\left(\sum_{k=1}^{\infty}\frac{q^{nk}}{\lambda_k}\right)&=
\sum_{k=1}^{\infty}\frac{q^{nk}}{\prod_{i=2}^{k}(1-q^i)}-\sum_{k=1}^{\infty}
\frac{q^{n(k+1)}}{\prod_{i=1}^{k}(1-q^i)} \\
&=\sum_{k=1}^{\infty}\frac{q^{nk}}{\prod_{i=2}^{k}(1-q^i)}-\sum_{k=2}^{\infty}
\frac{q^{nk}}{\prod_{i=1}^{k-1}(1-q^i)},
\end{align*}
where we have reindexed the second sum in the second line. Separating the
first term in the first sum, we get
\begin{align*}
\hphantom{(\lambda_{1}-q^n)\left(\sum_{k=1}^{\infty}\frac{q^{nk}}
{\lambda_k}\right)}
&=q^n+\sum_{k=2}^{\infty}\frac{q^{nk}}{\prod_{i=2}^{k}(1-q^i)}-
\sum_{k=2}^{\infty}\frac{q^{nk}}{\prod_{i=1}^{k-1}(1-q^i)} \\
&=q^n+\sum_{k=2}^{\infty}\frac{q^{nk}}{\prod_{i=2}^{k-1}(1-q^i)}
\left[\frac{1}{1-q^k}-\frac{1}{1-q}\right] \\
&< q^n .
\end{align*}
because all the terms in the summation are negative.  This establishes
(\ref{lemmaclaim}).

\begin{figure}
\begin{center}
\ifthenelse{\equal{\coloroption}{color}}
{\includegraphics[angle=0,scale=0.8]{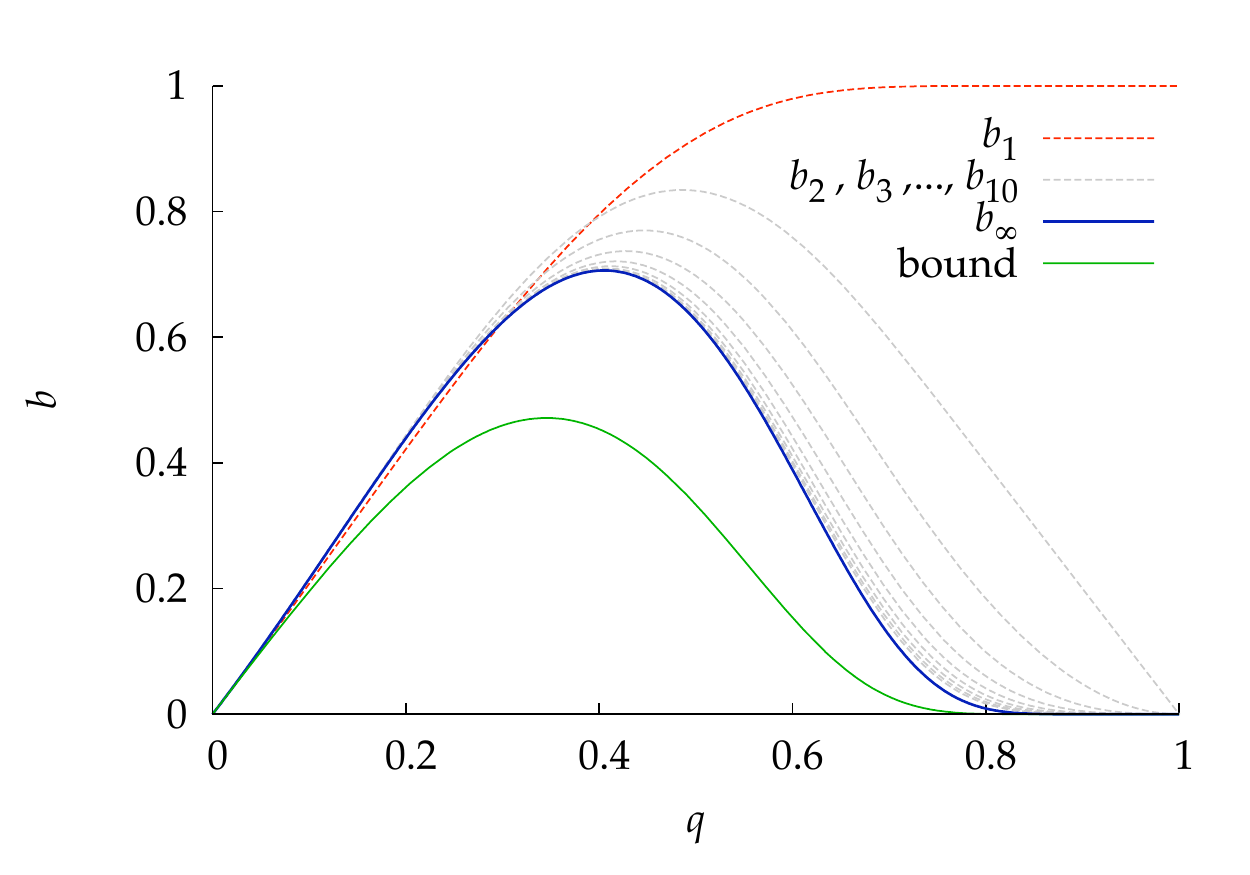}
\caption{\label{bvsq} 
  The offset terms $b_n(q)=\langle N_{n,1-q}\rangle-\kappa^2(q)n$, for
  $n=1$ (dashed red), $n=2,3,\ldots,10$ (dashed gray), and $n\to\infty$
  (solid blue).  The bound for $b(q)$ obtained by taking $n\to\infty$ in
  (\ref{bound}) is shown in green.}  }
{\includegraphics[angle=0,scale=0.8]{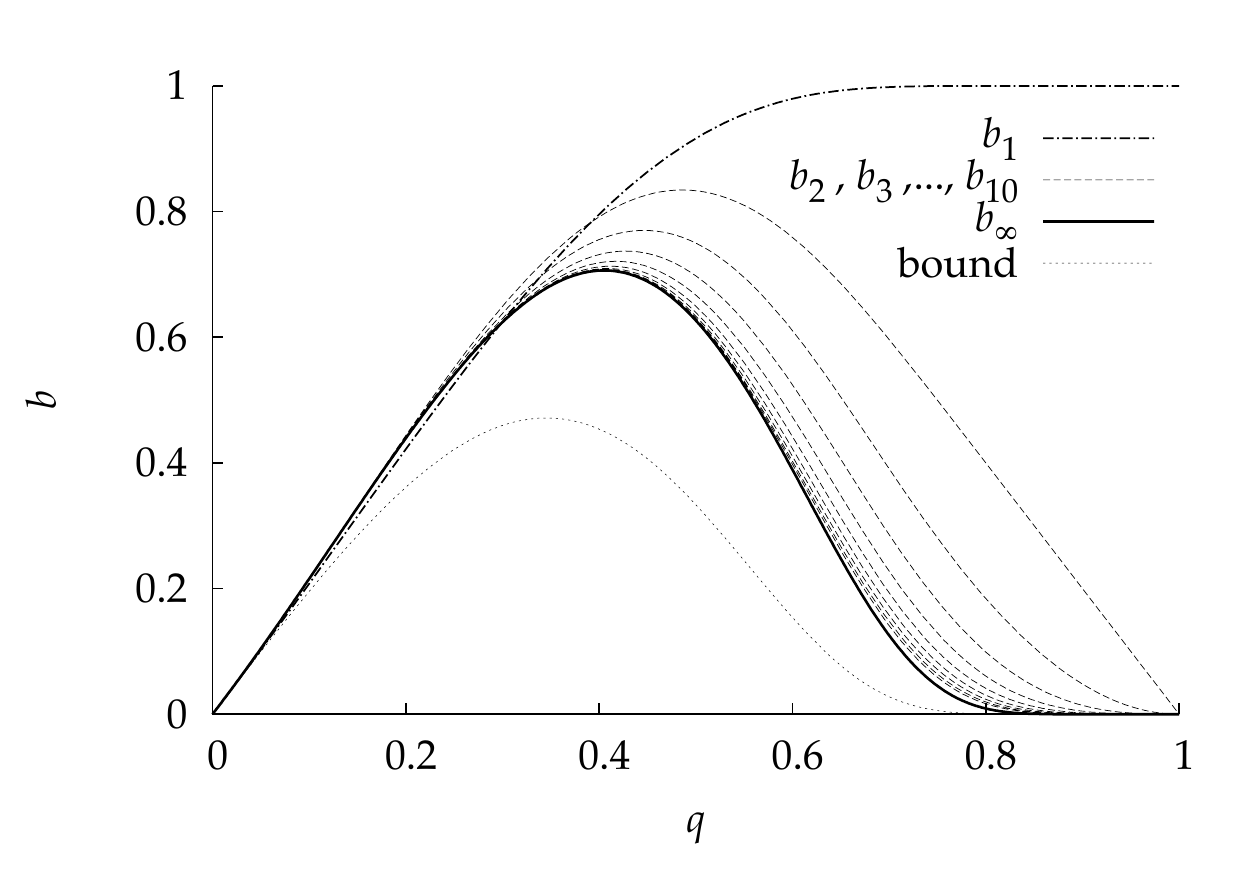}
\caption{\label{bvsq} 
  The offset terms $b_n(q)=\langle N_{n,1-q}\rangle-\kappa^2(q)n$, for
  $n=1,2,3,\ldots,10$, and $\infty$.  The bound for $b(q)$ obtained by
  taking $n\to\infty$ in (\ref{bound}) is also shown.
}
}
\end{center}
\end{figure}

Now we will show that (\ref{mainlemmaclaim}) follows from
(\ref{lemmaclaim}).  We use the following well-known identity
\cite{rademacher}, which holds for all complex $|x|<1$, $|z|<1$:
\begin{equation}\label{rade}
\prod_{m=1}^{\infty}(1-x^mz) =
   \left(\sum_{k=0}^{\infty}\frac{x^kz^k}{\prod_{i=1}^{k}(1-x^i)}\right)^{-1}. 
\end{equation}
For completeness, we include a proof of (\ref{rade}). Define
\begin{equation} \label{defoff}
f(x,z)=\frac{1}{\prod_{m=1}^{\infty}(1-x^mz)},\end{equation} 
and consider $f$ as a function of $z$ for fixed $|x|<1$.  
Since $\prod_{m=1}^{\infty}(1-x^mz)\neq0,$ we can write it as 
$\exp\left(\sum_{m=1}^{\infty}\log(1-x^mz)\right)$. This, in turn, may be
written  
\[\exp\left(-\sum_{m=1}^{\infty}\sum_{k=1}^{\infty}\frac{(x^mz)^k}{k}\right)
=\exp\left(-\sum_{k=1}^{\infty}\sum_{m=1}^{\infty}\frac{x^{mk}z^k}{k}\right).\]
This shows that $f$ is analytic throughout the unit disk $|z|<1$.
Therefore we can write $f$ as a power series in $z$:
$f(x,z)=\sum_{k=0}^{\infty}c_k(x)z^k$.  Now it is easy to see from
(\ref{defoff}) that $f(x,xz)=(1-xz)\,f(x,z)$, and this implies
$x^k\,c_k(x)=c_k(x)-x\,c_{k-1}(x)$.  Noticing that $c_0(x)=1$ and solving
this recursive relationship gives (\ref{rade}).

Finally, set $x=q$ and $z=q^{n-1}$, and apply (\ref{rade}) to the
definition of $\mu_n$ to obtain:
\[
\mu_n=\left(\sum_{k=0}^{\infty}\frac{q^{nk}}{\lambda_k}\right)^{-1}.
\]
Notice that $\mu_n=(1+S_n)^{-1}$, and rearrange (\ref{lemmaclaim}) to get
$\lambda_{n-1}\,S_n<(1+S_n)\,q^n$.  Putting these together, we have
\begin{equation*} 
\lambda_{n-1}-\kappa=\lambda_{n-1}(1-\mu_n) = 
\frac{\lambda_{n-1}\,S_n}{1+S_n} < q^n,
\end{equation*} 
as desired.
\end{proof}

Lemma 2 provides a very useful bound regarding the convergence of the
partial products of $\kappa$ to their limiting value.  We will use the
bound several times to prove statements that involve expressions of the
form $\lambda_{n-1}-\kappa$ or $1-\mu_n$.  While there are more efficient
ways to compute the Euler function\cite{sokal}, Lemma 2 also shows that
even the naive products converge reasonably well: the error is strictly
bounded by $q^n$, a considerable improvement---especially for $q$ close to
1---over the obvious estimate that the partial products are of order
$O(q^n)$.  Slater\cite{slater} first used (\ref{rade}) to compute the
Euler function numerically but did not observe its implications for the
naive products.

\begin{proof}[Proof of theorem.] First, we define the random variables 
\[
X_k=\left\{\begin{array}{ll} 1 & \text{$k$ is a post} \\ 0 & \text{$k$ is
      not a post} \end{array}\right. \qquad (k=1,2,\ldots,n).
\]
Then $N_{n,q} = \sum_{k=1}^n X_k$ and, by linearity, $\ev{N_{n,q}}
=\sum_{k=1}^n \ev{X_k} = \sum_{k=1}^{n}\text{Pr($k$ is a post)}$.  From the
proof of Theorem \ref{infty_thm}, we know
\[
\text{Pr($k$ is a post)} = \prod_{i=1}^{k-1}(1-q^i) \prod_{j=1}^{n-k}(1-q^j).
\] 
Substituting in the definitions of $\mu_k$ and $\kappa$ gives:
\begin{align*}
  \ev{N_{n,q}} &= \sum_{k=1}^{n}\frac{\kappa}{\mu_{k}}\,\frac{\kappa}{\mu_{n-k+1}}
  = \kappa^2 \sum_{k=1}^{n}\frac{1}{\mu_{k}\,\mu_{n-k+1}}.
\end{align*}
Define the ``offset" quantities $b_n$ according to (\ref{expect}):
\begin{equation}  \label{offset}
b_n =  \ev{N_{n,q}} - n\, \kappa^2
    = \kappa^2\sum_{k=1}^{n}\left[\frac{1}{\mu_{k}\,\mu_{n-k+1}}-1\right].
\end{equation}
First, we produce a lower bound on $b_n$.  As $x < -\!\log\left(1-x\right)$
for $x\in(0,1)$, we have 
\[
-\log\mu_{k} 
= -\sum_{i=k}^{\infty} \log(1-q^i)
> \sum_{i=k}^{\infty}q^i
= \frac{q^k}{1-q},
\]
which implies that 
\[\mu_{k} < \exp\left(-\frac{q^k}{1-q}\right).\]  From this we obtain
\[ \frac{1}{\mu_{k}\,\mu_{n-k+1}}-1 ~>~
\exp\left(\frac{q^k+q^{n-k+1}}{1-q}\right)-1 ~>~
 \frac{q^k+q^{n-k+1}}{1-q}.
\]
Substituting the previous equation into (\ref{offset}) gives
\begin{equation}  \label{bound}
b_n ~>~ \kappa^2\sum_{k=1}^{n} \frac{q^k+q^{n-k+1}}{1-q}
    ~=~ \frac{2\,q\,\kappa^2(1-q^n)}{(1-q)^2}, 
\end{equation}
where we have used the fact that the exponents $n-k+1$ and $k$ range over
the same set of values as $k=1,2,\ldots,n$.  This establishes that the
sequence $\{b_n\}_{n=2}^{\infty}$ is bounded below by a positive quantity
(see Fig. \ref{bvsq}).

Next, we prove that for every $n\geq2$ we have $b_n> b_{n+1}$. We
calculate the difference $b_n-b_{n+1}$:
\begin{align*}
\kappa^{-2}(b_n-b_{n+1})&=\sum_{k=1}^{n}\frac{1}{\mu_{k}\mu_{n-k+1}}-n-
\sum_{k=1}^{n+1}\frac{1}{\mu_{k}\mu_{n-k+2}}+n+1 \\
&=1-\frac{1}{\mu_1\mu_{n+1}}+\sum_{k=1}^{n}\frac{1}{\mu_k}
\left(\frac{1}{\mu_{n-k+1}}-\frac{1}{\mu_{n-k+2}}\right).
\end{align*}
Recall that summation by parts (see, e.g., Ref \cite{rudin}) says that for
general sequences $\{x_n\}$ and $\{y_n\}$, if we define
$X_n=\sum_{k=1}^{n}x_k$, we 
have $\sum_{k=1}^{n}x_ky_k=X_ny_n+\sum_{k=1}^{n-1}X_k(y_k-y_{k+1})$.  Taking
$x_k=1/{\mu_{n-k+1}}-1/{\mu_{n-k+2}}$ and $y_k=1/\mu_k$ in this formula, we
get (notice that $X_k$ is a telescoping sum):
\begin{align*}
&=1-\frac{1}{\mu_1\mu_{n+1}}+
\frac{1}{\mu_n}\left(\frac{1}{\mu_1}-\frac{1}{\mu_{n+1}}\right)+
\sum_{k=1}^{n-1}\left[\left(\frac{1}{\mu_{n-k+1}}-\frac{1}{\mu_{n+1}}\right)
\left(\frac{1}{\mu_{k}}-\frac{1}{\mu_{k+1}}\right)\right].
\end{align*}
The quantity $1/{\mu_{n-k+1}}-1/{\mu_{n+1}}$ in the first set of
parentheses in the sum can be made smaller by replacing $\mu_{n-k+1}^{-1}$
with $\mu_n^{-1}$ as $\mu_n$ is monotonically increasing in $n$.
Performing the remaining telescoping sum yields
\begin{align*}
&\geq 1-\frac{1}{\mu_1\mu_{n+1}}+\frac{1}{\mu_n}\left(\frac{1}{\mu_1}-
\frac{1}{\mu_{n+1}}\right)+\left(\frac{1}{\mu_{n}}-\frac{1}{\mu_{n+1}}\right)
\left(\frac{1}{\mu_{1}}-\frac{1}{\mu_{n}}\right)  \\
&= 1+\frac{2}{\mu_1}\left(\frac{1}{\mu_n}-\frac{1}{\mu_{n+1}}\right)-\frac{1}
{\mu_n^2}.
\end{align*}
Using the fact that $\kappa=\mu_1=\lambda_{n-1}\mu_n$ and expressing all
the denominators in terms of $\kappa$ gives
\begin{align*}
&=\frac{1}{\kappa^2}\left[\kappa^2-\lambda_{n-1}^2+2\lambda_{n-1}-
2\lambda_n\right].
\end{align*}
Factoring out a factor of $\lambda_{n-1}$ from each variable in the
numerator leads to
\begin{align*}
&=\frac{\lambda_{n-1}}{\kappa^2}
  \left[\lambda_{n-1}(\mu_n-1)(\mu_n+1)+2q^n\right]. 
\end{align*}
Now, because $\mu_n-1$ is negative, we can replace $\mu_n+1$ with 2 to make
the first term in brackets more negative.  Then we have
\begin{align} \nonumber
\kappa^{-2}(b_n-b_{n+1})&>\frac{2\lambda_{n-1}}{\kappa^2}
\left(\lambda_{n-1}(\mu_n-1)+q^n\right) \\ \label{monotone}
&=\frac{2\lambda_{n-1}}{\kappa^2}(\kappa - \lambda_{n-1}+q^n),
\end{align}
which is positive by the preceding lemma. This establishes that
$\{b_n\}_{n=2}^{\infty}$ is monotonically decreasing and hence must
converge to a (positive) limit $b$.

Now we will show that $ b_n < 1$ for all $n$.  From (\ref{offset}), we have
$0< b_1 < 1$ since $ b_1 = 1-\kappa^2$.  Also, $b_2=2(1-q-\kappa^2),$ so
$b_2<1$ whenever $q+\kappa^2>1/2$. We use the inequality $\kappa>1-q-q^2$,
which holds for all $0<q<1$ by (\ref{mainlemmaclaim}).  We get that
$q+\kappa^2>q+(1-q-q^2)^2$.  To verify that the right-hand side is greater
than $1/2$ for all $q$, notice that its derivative factors as
$-(1-2q)(1+4q+2q^2)$, which implies that it achieves its minimum at
$q=1/2$. At $q=1/2$, it is equal to $9/16$, which proves that
$b_2<1$. Since $\{b_n\}$ is monotonically decreasing for $n\geq 2$, this
establishes $b_n < 1 \:\forall\: n$.

Finally, we will prove (\ref{b_lim}). Define $B$ to be the right-hand side
of (\ref{b_lim})---with $B_k$ its $k$th partial sum---so that our goal is to
show $b=B$. Begin by writing $H=\lfloor n/2 \rfloor$ and $\delta_{\rm odd}
(n)=1$ if $n$ is odd and $0$ if $n$ is even. By splitting the
symmetric sum in (\ref{offset}), we have
  \begin{align}\nonumber 
    b_n &= 2\sum_{k=1}^{H}\left(\lambda_{k-1}\lambda_{n-k}-\kappa^2\right)+
    \delta_{\rm odd}(n)(\lambda_H^2-\kappa^2) \\ \nonumber &=
    2\kappa\sum_{k=1}^{H}\left(\lambda_{k-1}\mu_{n-k+1}^{-1}-\kappa\right)+
    \delta_{\rm odd}(n)(\lambda_H^2-\kappa^2). 
  \end{align}
  Now since $\mu_{n-k+1}^{-1}>1$, we can replace it with 1 and take the
  limit as $n\to\infty$ (so that the $\delta_{\rm
    odd}(n)(\lambda_H^2-\kappa^2)$ term goes to 0) to get that
  $b\geq B$.  Notice that $B$ is a positive series that is bounded above, and
  hence it must converge.  Now we look
  at the difference between the partial sums $b_n$ and $B_{H-1}$:
  \begin{align}\nonumber 
    b_n-B_{H-1} &=  2\kappa\sum_{k=1}^{H}\lambda_{k-1}\left(
        \mu_{n-k+1}^{-1}-1\right)+
      \delta_{\rm odd}(n)(\lambda_H^2-\kappa^2) \\ \nonumber
      &<  2\kappa\sum_{k=1}^{H}\left(
        \mu_{n-k+1}^{-1}-1\right)+
      \delta_{\rm odd}(n)(\lambda_H^2-\kappa^2)\\ \nonumber
      &<  2\kappa H\left(\mu_{H+1}^{-1}-1\right)+
      \delta_{\rm odd}(n)(\lambda_H^2-\kappa^2);
  \end{align}
in the last step we used to the facts that $n - \lfloor n/2 \rfloor \geq
\lfloor n/2 \rfloor$ and that $\mu_k$ is monotonically increasing in $k$.
Multiplying and dividing the first term by $\lambda_H$ and using the fact that
$\lambda_{H}(1-\mu_{H+1})<q^{H+1}$ by Lemma \ref{finite_lemma} yields:
  \begin{align}\nonumber 
   b_n-B_{H-1} &<  2\kappa H \frac{q^{H+1}}{\lambda_{H}\mu_{H+1}}+
      \delta_{\rm odd}(n)(\lambda_H^2-\kappa^2) \\ \nonumber
      &=  2 H q^{H+1}+
      \delta_{\rm odd}(n)(\lambda_H^2-\kappa^2).
  \end{align}    
  Taking the limit as $n\to\infty$ of both sides gives that $b\leq
  B$, which completes the proof.
\end{proof}

\section{Conclusions}

In this paper we considered random graph orders on $n$ elements.  Consistent
with the known fact that infinite random graph orders have infinitely many
posts \cite{Alon}, we have shown in Theorem 2 that the mean value of the number
$N_{n,q}$ of posts grows linearly in $n$ with a mean separation between
posts approaching
\[
\kappa^{-2}(q) = \bigg[\prod_{k=1}^\infty (1-q^k) \bigg]^{-2}\!\!\!\!\!\!,
\]
the same mean spacing between posts as in a random graph order over
$\mathbb{Z}$. 

In a finite random graph order over $\{1,2,\ldots,n\}$, however, the actual
value of $\langle N_{n,q}\rangle$ is not exactly proportional to $n$ but
includes a small positive offset $b_n(q)$.  This offset stems from the fact
that the first and the last few elements are appreciably more likely than
the other ones to be posts.  Thus, the functions $b_n(q)$ (shown in
Fig.\ \ref{bvsq}) quantify the edge effects, with $\frac12\,b_n(q)$
corresponding to the contribution of each of the two ends of the poset.  As
Theorem 2 and Fig.\ \ref{bvsq} show, that contribution does not vanish even
in the $n\to\infty$ limit.  Although we do not have an analytical
expression or bound for the standard deviation of $N_{n,q}$ at this point,
our numerical results suggest that---consistent with intuition---it is
proportional to $\sqrt{n}$ (see Fig.\ \ref{f2}). 

Overall, this result does not significantly affect the number of posts or
the size of the inter-post region in a random graph order.  As mentioned in
the Introduction, the latter is one of the first quantities one considers
when estimating how viable such posets are as discrete models for
spacetime.  Therefore, as far as allowing inter-post regions of a random
graph order to grow large, the only condition we need to impose is that $q$
be sufficiently close to 1; equivalently, the probability $p=1-q$ of
linking two elements in the random graph must be small enough. The size of
the inter-post regions is not the only condition we would impose for a
poset to be manifoldlike; we are also studying the effects of other
requirements and will report our results on them separately. One
interesting byproduct of the work reported here, however, is the bound
(\ref{mainlemmaclaim}) in Lemma 2. This bound was useful for us in proving
the assertions in Theorem \ref{finite_thm}, but, because of the importance
of the Euler function, it may be useful in other contexts as well.

\section*{Acknowledgments}
The authors would like to thank David Rideout and Graham Brightwell for 
helpful suggestions.

\bibliography{post_asymptotics}
\bibliographystyle{cpc}

\end{document}